\documentclass[11pt]{amsart}

\usepackage{amsfonts, amssymb, amscd}

\usepackage{verbatim}
\usepackage{amssymb}
\usepackage{mathrsfs}
\usepackage{graphicx}
\usepackage{cite}
\usepackage[all]{xy}
\usepackage{tikz}
\usepackage[toc,page]{appendix}
\usepackage{mathtools}

\usepackage{hyperref}

\newcommand{\Zz}{\mathbb{Z}}
\newcommand{\Cc}{\mathbb{C}}
\newcommand{\Pp}{\mathbb{P}}
\newcommand{\Rr}{\mathbb{R}}

\newcommand{\zz}{\mathbf{z}}

\newcommand{\0}{\mathbf{0}}

\newcommand{\Ver}{\operatorname{Vert}}
\newcommand{\Conv}{\operatorname{Conv}}

\newcommand{\Hom}{\operatorname{Hom}}
\newcommand{\rk}{\operatorname{rank}}

\newcommand{\Coh}{\operatorname{Coh}}
\newcommand{\D}{\operatorname{D}}
\newcommand{\Dcoh}{\operatorname{Dcoh}}
\newcommand{\Kcoh}{\operatorname{Kcoh}}
\newcommand{\Acoh}{\operatorname{Acoh}}

\newcommand{\la}{\langle}
\newcommand{\ra}{\rangle}

\newcommand{\Aa}{\mathcal{A}}
\newcommand{\Bb}{\mathcal{B}}

\newcommand{\Oo}{\mathcal{O}}

\newcommand{\Ee}{\mathcal{E}}
\newcommand{\Ff}{\mathcal{F}}

\newcommand{\Ss}{\mathcal{S}}
\newcommand{\Yy}{\mathcal{Y}}

\newtheorem{theorem}{Theorem}[section]

\newtheorem{proposition}[theorem]{Proposition}
\newtheorem{definition}[theorem]{Definition}

\newtheorem{remark}[theorem]{Remark}
\newtheorem{conjecture}[theorem]{Conjecture}

\numberwithin{equation}{section}

\begin{document}

\title[General Clifford double mirrors]{On derived equivalence of general Clifford double mirrors}

\begin{abstract}
We show that general Clifford double mirrors constructed in \cite{BL16} are derived equivalent.
\end{abstract}

\author{Zhan Li}
\address{Beijing International Center for Mathematical Research\\ Peking University, 
Beijing 100871, China} \email{lizhan@math.pku.edu.cn}

\maketitle

\tableofcontents

\section{Introduction}\label{sec: introduction}
Mirror symmetry originated from the observation in physics that different Calabi-Yau threefolds may provide (physical) compactifications of dual string theories. In particular, all physical predictions of the two theories are the same. Such Calabi-Yau varieties are called mirror pairs. Mathematically, mirror symmetry is reflected by the relations between mirror pairs on Hodge numbers, derived and Fukaya categories, Gromov-Witten invariants, etc.

\medskip

In recent years, much of the attention has been drawn to the double mirror phenomenon, that is, two Calabi-Yau varieties are both mirror to a same Calabi-Yau variety (also called multiple mirror phenomenon in \cite{CK99}). In this scenario,  properties of double mirror Calabi-Yaus can be read off from mirror symmetry predictions. For example, their $(p,q)$-stringy Hodge numbers should be same as they should both equal to the $(n-p, q)$-stringy Hodge number of their common mirror; their derived categories are expected to be equivalent, because according to homological mirror symmetry conjecture \cite{Kon94}, they are both equivalent to the Fukaya category of their mirror.

\medskip

These properties have been studied for various known double mirror pairs. For the Batyrev-Borisov double mirrors, the equality of their stringy Hodge numbers is a consequence of the main result of \cite{BB96} (Theorem 4.15), their derived equivalence are confirmed for the corresponding stacks in \cite{FK16} (Theorem 6.3), and their birationality has been proved under some mild assumptions in \cite{Li13} (Theorem 4.10). The analogous results for Berglund-H\"ubsch-Krawitz mirrors have been established in \cite{CR11, FK16b, Sho12, Bor13, Kel13, Cla13}.

\medskip

Mirror symmetry for Calabi-Yau varieties has been generalized to Landau-Ginzburg model (LG-model) and Calabi-Yau (or Fano) correspondence. LG-model $(X, w)$ consists a variety $X$ and a regular function $w$ on it which is called potential. On the level of derived category, LG-model corresponds to derived matrix factorization category which is equivalent to (relative) singular category of zero locus of $w$. There also exits double mirror phenomenon in this case. In the Givental's LG/Fano setting, Prince shows that for Fano complete intersections in toric varieties, certain Laurent polynomial multiple mirrors are related by a mutation (\cite{Pri}, c.f. \cite{CKP15} Theorem 5.1). A similar result also appears in \cite{HD15} Theorem 2.24.

\medskip

Besides above physics considerations, there are a number of sporadic examples \cite{Muk88, Kuz08, CDHPS10, Add09,  CT14} in the literature involving derived equivalent noncommutative varieties which postulates their connections with double mirror phenomenon. It is this observation that motives our work \cite{BL16} to uncover the toric geometry underlings of such examples and relates them to mirror symmetry. 

\medskip

In \cite{BL16}, we work in a slightly more general setting of Batyrev-Borisov construction where we consider a pair of reflexive cones (see Definition \ref{def: reflexive cones}). It has been known that a decomposition of the degree element of a reflexive cone with coefficients $1$ will result in Batyrev-Borisov double mirrors (i.e. two Calabi-Yau complete intersections as double mirrors in the Batyrev-Borisov construction). We generalize this by allowing coefficients $1/2$ in the decomposition of degree element, and construct a noncommutative variety (stack) $(\Ss, \Bb_0)$ associated to it, where $\Ss$ is a complete intersection in a Fano toric variety and $\Bb_0$ is a noncommutative sheaf of algebra. Such construction depends on a parameter $r$ which ``counts'' how many terms with coefficients $1/2$. When $r=0$, then $\Bb_0=0$, and we are back to the Batyrev-Borisov situation where the noncommutative variety is just a commutative Calabi-Yau variety. However, when $r>0$, $\Bb_0$ is non trivial and $\Ss$ is no longer Calabi-Yau. In this case, $(\Ss, \Bb_0)$ can be viewed as a noncommutative Calabi-Yau variety and we call it (general) Clifford mirror. One main result of \cite{BL16} (Theorem 6.3) is that when $r$ archives its extreme values (i.e. the complete intersection and pure Clifford mirror cases), the corresponding double mirrors are derived equivalent. We conjectured further that for general $r$, the corresponding general Clifford double mirrors should also pass the double mirror tests. Especially, under some appropriate conditions, no matter which $r$ is chosen, they are all derived equivalent (\cite{BL16} Conjecture 7.5). The goal of present paper is to give an affirmative answer to that conjecture (Theorem \ref{thm: main}).

\medskip

We briefly discuss the content of each section. In Section \ref{sec: construction}, we explain the construction of general Clifford mirror in \cite{BL16}. During that course, we give necessary combinatoric definitions and fix the notation. In Section \ref{sec: derived equivalence}, we give a proof for derived equivalence of such Clifford double mirrors. It relies on Shipman, Isik, Hirano's result on Kn\"orrer periodicity \cite{Shi12, Isi13, Hir16} and homological variations of GIT quotients \cite{BFK12, HL15}. The section ends up with remarks on possible ways to check other relations between general Clifford double mirrors. In Section \ref{sec: examples}, we give examples of general Clifford double mirrors and heuristic explanations of the derived equivalence in such situations. 

\medskip
\noindent{\it Acknowledgements.} This paper can be viewed as a follow-up to the joint paper \cite{BL16} with Lev Borisov. The author thanks Lev Borisov for reading the manuscript and providing useful suggestions. In fact, Proposition \ref{prop: base point free} is due to him which removes an unnecessary assumption in the first version. The author also benefits from constructive comments of Yongbin Ruan and discussions with Feng Qu and Evgeny Mayanskiy. The work is partially supported by Chenyang Xu's grant.

\section{The construction of general Clifford double mirrors}\label{sec: construction}

In this section, we recall the construction of general Clifford double mirrors given in \cite{BL16} (see Section 7), and fix the notation used throughout the paper. We begin with some combinatoric definitions.

\medskip

Let $M \cong \Zz^{\rk M}$ be a lattice and let $N:=\Hom_\Zz(M, \Zz)$ be its dual lattice. The natural pairing is given by
\[\la , \ra : M \times N \to \Zz.\]
Let $M_\Rr:= M \otimes_\Zz \Rr, N_\Rr:= N \otimes_\Zz \Rr$ be the $\Rr$-linear extensions of $M, N$. The pairing can be $\Rr$-linearly extended, and we still use $\la , \ra$ to denote this extension.

\begin{definition}
A \emph{rational polyhedral cone} $K \subset M_\Rr$ is a convex cone generated by a finite set of lattice points. We assume that $K\cap (-K)=\{\0\}$.
We call the first lattice point of a ray $\rho$  of $K$ a \emph{primitive element} or a \emph{lattice generator} of $\rho$.
\end{definition}

\begin{definition}[\cite{BB94}]\label{def: reflexive cones}
A full-dimensional rational polyhedral cone $K \subset M_\Rr$ is called a \emph{Gorenstein cone} if all the primitive elements of its rays lie on some hyperplane $\la - , \deg^\vee \ra =1$ for some \emph{degree element} $\deg^\vee$ in $N$.
A Gorenstein cone $K \subset M_\Rr$ is called a \emph{reflexive Gorenstein cone} iff its \emph{dual cone} $K^\vee: =\{y \mid \la x, y \ra
\geq 0 ~{\rm for~all}~ x \in K\}$ is also a Gorenstein cone with respect to the dual lattice $N$.
\end{definition}

\begin{definition}
For a pair of dual reflexive Gorenstein cones $(K,K^\vee)$, the pairing of their two degree elements
$\la \deg, \deg^\vee \ra = k$ is called the \emph{index of the pair}. The index is
always a positive integer.
\end{definition}

We consider a pair of reflexive Gorenstein cones $K$ and $K^\vee$
in lattices $M$ and $N$ with degree elements $\deg\in K$
and $\deg^\vee\in K^\vee$  respectively.  Suppose the index of this pair of Gorenstein cones is $k=\la \deg,\deg^\vee\ra$.
In addition, we consider a generic coefficient function
\begin{equation}\label{eq: coefficient function}
c:K_{(1)}\to \Cc,
\end{equation} where $K_{(1)}: = \{m \in K \cap M \mid \la m, \deg^\vee \ra = 1 \}$.

\medskip

As explained in \cite{BL16} (see Section 2 and 7), a decomposition of the degree element $\deg^\vee$ as a summation of lattice elements encompasses the data for toric double mirrors. For example, if $\deg^\vee$ can be expressed as a linear combination of elements in $K^\vee_{(1)}$ with coefficients $1$, then all such decomposition will result in the Batyrev-Borisov double mirrors (see \cite{Li13} Theorem 3.4); and if the linear combination is of coefficients $1/2$, they will correspond to the pure Clifford double mirrors. This justifies the following consideration of both types of coefficients. 

\medskip
Suppose that we have
\begin{equation}\label{eq: decomposition of degree}
\deg^\vee =  \frac{1}{2}(s_1 + \cdots + s_{2r}) + t_1 + \cdots
+ t_{k-r}
\end{equation}
for some $0\leq r\leq k$, with $s_i, t_j \in K^\vee \cap N$. The $(k+r)$ elements $s_i$ and $t_j$ are assumed to be
linearly independent. In addition, there should exist (and be chosen) a
\emph{regular simplicial} fan $\Sigma$ (see \cite{CLS11} Definition 15.2.8) with support  $K^\vee$
such that the following \emph{centrality condition} holds (see \cite{BL16} (7.1)).
\begin{equation}\label{eq: centrality}
\emph{All maximum dimensional cones of $\Sigma$
contain $\{s_i,t_j\}$ as ray generators.} 
\end{equation}


We follow the usual process to define toric stacks (see \cite{BCS05}). First, we consider the Cox open subset $U_\Sigma$ of $\Cc^{K^\vee_{(1)}}$
of functions
\begin{equation}\label{eq: Cox open set}
{\bf z}:K^\vee_{(1)}\to \Cc
\end{equation}
such that the preimage of $0$ is a subset of $\Sigma$. For simplicity, let $\{s_i; t_j\}$ denote the set 
$\{s_1,\ldots, s_{2r}, t_1, \ldots, t_{k-2r}\}$. We can similarly consider the
subset
\[
U_{\overline\Sigma} \subset \Cc^{K^\vee_{(1)}-\{s_i,t_j\}}
\]
that corresponds to the stacky fan $\overline\Sigma$ for the group
\begin{equation}\label{eq: bar N}
\overline N  = N /\Zz s_1 + \cdots + \Zz s_{2r}+ \Zz \deg^\vee+ \Zz
t_1+\cdots+ \Zz t_{k-r}. 
\end{equation}
We have
\[
U_\Sigma = U_{\overline\Sigma} \times \Cc^{2r}\times \Cc^{k-r},
\]
where the last two components correspond to coordinates of ${\bf z}$ at $s_i$ and $t_i$.

\medskip

There is a group $\hat G$ defined by
\begin{equation}\label{eq: G hat}
\hat G:=\{ { \bf \lambda}:K^\vee_{(1)}\to \Cc^*\Big\vert
\prod_{n\in K^\vee_{(1)}} \lambda(n)^{\la m,n\ra} = 1,~{\rm for~all~}m\in {\rm Ann}(\deg^\vee)\},
\end{equation} where ${\rm Ann}(\deg^\vee):=\{m \in M \mid \la m, \deg^\vee\ra =0\}$. The group $\hat G$ acts naturally on $U_{\Sigma}$. It has a subgroup $H \subset \hat G$ which is isomorphic to $\Cc^*$ with action
\begin{equation}\label{eq: H = C*}
\lambda(s_i) = t, \lambda(t_i) = t^2, \lambda(v)=1,{\rm ~for~all~}v\in K^\vee_{(1)}-\{s_1,\ldots,s_{2r},
t_1,\ldots,t_{k-r}\}.
\end{equation}
It is shown in \cite{BL16} Section 5 that the toric DM stack corresponding to $(\overline N,\overline\Sigma)$ can be realized as the quotient of $U_{\overline\Sigma}$ by 
\begin{equation}\label{eq: G-overline}
\overline G= \hat G/H.
\end{equation}

Moreover, there exists (non-unique) isomorphism $\hat G \cong \overline G \times H$ (\cite{BL16} Remark 5.1).

\medskip
There is a $\overline G$-invariant polynomial which will be called \emph{potential} in the sequel
\begin{equation}\label{eq: potential}
C({\bf z})=\sum_{m\in K_{(1)} }c(m)  \prod_{n\in
K^\vee_{(1)}} \zz(n)^{\la m,n\ra}.
\end{equation}
The potential $C(\zz)$ is of total degree $2$ with respect to $H \cong \Cc^*$. It can be further written as 
\begin{equation}\label{eq: decomposition of C}
C(\zz) = C_1(\zz) + C_2(\zz),
\end{equation} where $C_1(\zz)$ is a linear term in ${\bf z}(t_i)$, and $C_2(\zz)$ is a quadratic term in $\zz(s_i)$. More precisely,
\begin{equation}\label{eq: C_1}
C_1(\zz) =  \sum_{m\in K_{(1)} \cap {\rm
Ann}(s_1,\ldots,s_{2r})} c(m) \prod_{n\in  K^\vee_{(1)}} \zz(n)^{\la
m,n\ra},
\end{equation} 
\begin{equation}\label{eq: C_2} 
C_2({\bf z}) = \sum_{m\in K_{(1)} \cap {\rm
Ann}(t_1,\ldots,t_{k-r})} c(m) \prod_{n\in  K^\vee_{(1)}} \zz(n)^{\la
m,n\ra}.
\end{equation}
If we let 
\begin{equation}\label{eq: f_i}
\begin{split}
f_i &= \sum_{m\in K_{(1)} \atop \la m,
t_i\ra =1}c(m) \prod_{n\in  K^\vee_{(1)} - \{t_i\}} \zz(n)^{\la
m,n\ra}\\
& = \sum_{m\in K_{(1)} \atop \la m,
t_i\ra =1}c(m) \prod_{n\in  K^\vee_{(1)} - \{t_i; s_{j}\}} \zz(n)^{\la
m,n\ra}
\end{split}
\end{equation} then 
\begin{equation}\label{eq: decomposition of C_1}
C_1(\zz) = \sum_{i=1}^{k-r} \zz(t_i) f_i.
\end{equation}
Indeed, by the decomposition \eqref{eq: decomposition of degree} and the choice of $m \in K_{(1)}$, we see that if $\la m, t_i \ra =1$, then $\la m, s_l \ra = \la m, t_j \ra = 0$ for all $s_l$ and $t_j \neq t_i$. It is straightforward to verify that $C_2(\zz)$ and $\zz(t_i) f_i$ are both $\hat G$-semiinvariant section with character 
\begin{equation}\label{eq: chi}
\chi(\lambda) = \prod_{n \in K^\vee_{(1)}}\lambda(n)^{\la \alpha, n \ra},
\end{equation} where $\alpha \in M$ is any lattice point satisfying $\la \alpha, \deg^\vee \ra=1$. 

\medskip

One can check that $\hat G$ is generated by $H$ and another subgroup 
\[G': = \{ { \bf \lambda}:K^\vee_{(1)}\to \Cc^*\Big\vert
\prod_{n\in K^\vee_{(1)}} \lambda(n)^{\la m,n\ra} = 1,~{\rm for~all~}m\in M\}.\]

By definition of $G'$, $\zz(t_i)f_i$ is $G'$-invariant, and for any $\lambda \in H \cong \Cc^*$ (see \eqref{eq: H = C*}), we have $\lambda \cdot (\zz(t_i)f_i) = \lambda(\zz(t_i))(\zz(t_i)f_i) = \chi(\lambda) (\zz(t_i)f_i)$, hence 
\begin{equation}\label{eq: chi = lambda}
\chi(\lambda) = \lambda(\zz(t_i)), 1 \leq i \leq k-r.
\end{equation} In particular, this shows that $f_i$ is $\hat G$-invariant.

\medskip

We use $f_i$ to define an
intersection $ Y\subset U_{\overline\Sigma}$ by
\begin{equation}\label{eq: Y}
Y=\bigcap_{i=1}^{k-r} \{f_i = 0\}.
\end{equation} It can be viewed as the zero locus of the section $f:=(f_1, \ldots, f_{k-r}) \in H^0(U_{\Sigma}, \oplus_{i=1}^{k-r} \Oo_{U_{\Sigma}})$. Moreover, $Y$ is $\overline G$-invariant, and will be proved to be a complete intersection in Proposition \ref{prop: regularity}. 

\medskip

We then define $\Ss$ as the quotient stack $[Y/\overline G]$. The sheaf of
(even part of) Clifford algebras  $\Bb_0$ on $\Ss=[Y/\overline G]$ is defined by using the quadratic part $C_2({\bf
z})$ of $C({\bf z})$. We formulate its definition as follows.

\begin{definition}[see \cite{Kuz08} Section 3]\label{def.cliff.gen}
Let the quadric $C_2(\zz)$ be a section of ${\rm Sym}^2 (\oplus_{i=1}^{2r} \Oo_Y \cdot z^\vee_i)^\vee$, where $z^\vee_i$ are noncommuting variables. Then the even part of sheaf of Clifford algebra over $Y\subseteq U_{\overline\Sigma}$ is defined to be the noncommutative locally constant sheaf of algebra
\[
\Big(\Oo_{Y}\{z^\vee_1,\ldots,z^\vee_{2r}\}/\la (v \otimes v' + v' \otimes v - 2 C_2(v,v'),~~{\rm for~all~} v,v' \in \oplus_{i=1}^{2r} \Cc z_i^\vee)\Big)_{even},
\] where \emph{even} refers to elements of even degrees in $z_i^\vee$. It has a natural $\hat G$-equivariant structure which descends to $\hat G/H={\overline G}$ (see \cite{BL16} Remark 7.4). Then $\Bb_0$ on $[Y/\overline G]$ is defined to be above even part of sheaf of Clifford algebra over $Y\subseteq U_{\overline\Sigma}$ with $\overline G$-equivariant structure.
\end{definition}

\begin{definition}[\cite{BL16} Section 7]\label{def: general Clifford double mirrors}
For arbitrary $r, 0 \leq r \leq 2k$, we call noncommutative pair $(\Ss, \Bb_0)$ a \emph{general Clifford mirror}. The sheaf of Clifford algebra $\Bb_0$ should be viewed as the structure sheaf of this general Clifford mirror.
\end{definition}

This definition generalizes the \emph{pure Clifford mirror} (i.e. $r=2k$ case) considered in \cite{BL16} Section 5 to a complete intersection in a toric stack. The main conjecture of \cite{BL16} is that for a fixed reflexive Gorenstein cone and varied $r$, all the decompositions \eqref{eq: decomposition of degree} will give double mirrors. In terms of homological mirror symmetry, this can be formulated as follows.
\begin{conjecture}\label{conj: double mirrors}
Under the centrality and appropriate flatness assumptions, the bounded derived categories of coherent sheaves on $(\Ss,\Bb_0)$ are all equivalent.
\end{conjecture}

This conjecture holds in the boundary cases.

\begin{theorem}[Theorem 6.3 \cite{BL16}]\label{thm: extreme cases}
The conjecture \ref{conj: double mirrors} is true when $r$ choose $0$ or $2k$.
\end{theorem}

\begin{remark}
For $r=0$,  the noncommutative varieties $(\Ss, \Bb_0)$ are actually the Batyrev-Borisov Calabi-Yau varieties. Hence, Conjecture\ref{conj: double mirrors} can be viewed as a generalization of Batyrev-Nill's conjecture (\cite{BN08} Conjecture 5.3) on Batyrev-Borisov double mirrors which is answered affirmatively by Favero and Kelly \cite{FK16} (Theorem 6.3).
\end{remark}

The goal of this paper is to show that for all $r, 0 \leq r \leq 2k$, Conjecture \ref{conj: double mirrors} still holds. This is a strong evidence that the construction $(\Ss, \Bb_0)$ are indeed double mirrors, and we expect that they should also pass other tests of mirror symmetry.

\section{Derived equivalence of general Clifford  double mirrors}\label{sec: derived equivalence}

\subsection{Derived factorization category and Hirano's result}
The technical tools used to prove Theorem \ref{thm: main} are the relations of derived categories between variation of GIT (see \cite{BFK12, HL15}) and Hirano's analogy of Isik and Shipman's result (see \cite{Shi12, Isi13}) in the matrix factorization categories \cite{Hir16}. First recall the definition of derived matrix factorization categories of \cite{Pos11, EP15} (see \cite{Hir16} Definition 2.10).

\begin{definition}\label{def: factorization}
Let $X$ be a scheme, and $G$ be an affine algebraic group acting on $X$. Let $\chi: G \to G_m$ be a character of $G$, and $W: X \to \mathbb{A}^1$ be a $\chi$-semiinvariant function. A \emph{factorization} $F$ of data $(X, \chi, W, G)$ is a sequence
\[
F= \left(F_1 \xrightarrow{\phi^F_1}  F_0 \xrightarrow{\phi^F_0} F_1(\chi)\right),
\] where $F_i$ are $G$-equivariant coherent sheaves on $X$ and $\phi_i$ are $G$-equivariant homomorphisms. They satisfy the relations $\phi^F_0\circ \phi^F_1 = W \cdot {\rm id}_{F_1}, \phi^F_1(\chi) \circ \phi^F_0 = W \cdot {\rm id}_{F_0}$. A morphism of factorizations $g: E \to F$ is a pair of morphisms $(g_1, g_0)$ that commute 
with $\phi_i^E$ and $\phi_i^F$. We use ${\Coh}_G(X, \chi, W)$ to denote this abelian category of factorizations.
\end{definition} 

There also exists a notion of chain homotopy between morphisms in $\Coh_G(X, \chi, W)$ and we
let $\Kcoh_G(X, \chi, W)$ be the corresponding homotopy category. One can define a natural translation and
cone construction in $\Kcoh_G(X, \chi, W)$. These give a triangulated category structure on the homotopy
category. Let $\Acoh_G(X, \chi, W)$ be the smallest thick subcategory of $\Kcoh_G(X, \chi, W)$ containing all totalizations of
short exact sequences from $\Coh_G(X, \chi, W)$ (see \cite{Hir16} Section 2).

\begin{definition}\label{def: derived factorization category}
The \emph{derived factorization category} of data $(X, \chi, W, G)$ is defined as Verdier quotient
\[
\Dcoh_G(X, \chi, W):= \Kcoh_G(X, \chi, W) / \Acoh_G(X, \chi, W).
\]
\end{definition}

To state Hirano's result (\cite{Hir16} Theorem 4.2), let us fix the following notation. For consistency, it appears slightly different here than that in \cite{Hir16}. 

\medskip

Let $U$ be a smooth quasi-projective variety, $G$ be an affine algebraic
group acting on $U$, $\chi : G \to \mathbb{G}_m$ be a character, and $C_2 : U \to \mathbb{A}^1$
be a $\chi$-semiinvariant regular function. Suppose there is a $G$-equivariant locally free sheaf $\Ee$ over $U$, and a $G$-invariant section $f \in H^0(U, \Ee^\vee)$. Let $Z$ be the zero locus of $f$. We call $f$ to be a \emph{regular} section if the codimension of $Z$ in $U$ is $\rk \Ee$ (see \cite{Hir16} Section 4).  Set $\Ee(\chi)=\Ee \otimes \Oo(\chi)$ for the character $\chi$, and $V_U(\Ee(\chi)):= Spec({\rm Sym}^\bullet (\Ee(\chi)^\vee))$ the corresponding vector bundle with induced $G$-action. Let $q: V_U(\Ee(\chi)) \to U$ and $p: V(\Ee(\chi))|_Z \to Z$ be two natural projections. The regular section $f$ induces a $\chi$-semiinvariant regular function $C_1: V_U(\Ee(\chi)) \to \mathbb{A}^1$. Locally, this means that we associate $f=(f_1, \ldots, f_{\rk \Ee})$ to a function $C_1 = \sum_{i=1}^{\rk \Ee} z_i f_i$, where $z_i$ are indeterminates such that $g \in G$ acting on $z_i$ equals to $\chi(g) z_i$.

\begin{theorem}[\cite{Hir16} Theorem 4.2]\label{thm: Hirano}
Notation as above, assume $C_2|_Z$ is flat and $f \in H^0(U, \Ee^\vee)$ is a regular section, then there is an equivalence
\[
\Dcoh_G(Z, \chi, C_2|_Z)\cong \Dcoh_G(V_U(\Ee(\chi)), \chi, q^*C_2+C_1).
\]
\end{theorem}

\subsection{Proof of the main theorem}
Suppose there are two decompositions as \eqref{eq: decomposition of degree}
\[\begin{split}
\deg^\vee &= t_1 + \ldots + t_{k-r}+\frac 1 2 (s_1 + \ldots + s_{2r})\\
&=\tilde t_1 + \ldots + \tilde r_{k-l}+\frac 1 2 (\tilde s_1 + \ldots + \tilde s_{2l}),
\end{split}\]
and respective regular triangulations $\Sigma, \tilde \Sigma$ which satisfy centrality \eqref{eq: centrality}. Then, there exist general Clifford double mirrors $(\Ss, \Bb_0)$ and $(\tilde \Ss, \tilde \Bb_0)$ as constructed in Section \ref{sec: construction}. In order to show their derived equivalence, we first employ a derived version of Cayley trick, that is, we associate each derived category to a derived matrix factorization category with common potential $C(\zz)$; then a homological variation of GIT argument for derived categories will establish the desired equivalence. 

\medskip

Let us first work with $\deg^\vee = t_1 + \ldots + t_{k-r}+\frac 1 2 (s_1 + \ldots + s_{2r})$. The other decomposition can be treated analogously. We will show that the regularity assumption in Theorem \ref{thm: Hirano} is satisfied for the generic $Y$ in the general Clifford mirror construction. The following proposition is due to Lev Borisov.

\begin{proposition}\label{prop: base point free}
For each $1 \leq i \leq k-r$, the linear system 
\[L_i:=\{\sum_{m \in K_{(1)}, \la m, t_i \ra =1} c(n) \prod_{n \in K_{(1)}-\{s_i; t_j\}}\zz(n)^{\la m, n \ra} \mid c(n) \in \Cc\}
\]  is base point free on $U_{\overline \Sigma}$.
\end{proposition}

\begin{proof}
Let $F_m = \prod_{n \in K_{(1)}-\{s_i; t_j\}}\zz(n)^{\la m, n \ra}$ for $m \in \{m \in K_{(1)}\mid \la m, t_i \ra =1\}$. It suffices to show that there is no common zeros for all monomial functions $F_m$. $F_m$ is non-zero on $(\Cc^*)^{\#\Ver(\overline \Sigma)} \subset U_{\overline \Sigma}$, hence we only need to consider zero loci on the boundary divisors $\tilde D_n :=\{\zz(n) = 0\} \subset U_{\overline \Sigma}$ for $n \in K_{(1)}-\{s_i; t_j\}$. By direct computation, the zero divisor of $F_m$ is 
\[
{\rm div}_0(F_m) = \sum_{n \in K_{(1)}-\{s_i; t_j\}} \la m, n \ra \tilde D_{n}.
\] Assume the opposite.
If there were a closed point $z\in U_{\overline \Sigma}$ on which all $F_m$ are zero,
we may assume that the set of zeroes of $z$ lies in the maximum cone
$\bar\sigma$ of ${\overline \Sigma}$. This implies that the pairing $\la m,n \ra$ is positive for
at least
one $n$ which is the preimage of a generator of ${\overline \Sigma}$ under $\Sigma \to \overline \Sigma$. Indeed,
otherwise, the corresponding monomial $F_m$ does not involve any of the
coordinates that vanish on $z$.

\medskip

Let $n_1,\ldots, n_d$ be the preimages in $K^\vee_{(1)}$ of the generators of
$\bar\sigma$. Then the preimage $\sigma$ of $\bar\sigma$ is the maximum
cone with generators 
\[n_1,\ldots,n_d,s_1,\ldots, s_{2k}, t_1,\ldots,
t_{r-k}.
\] Consider the facet $\theta$ of this cone generated by all of the
above elments, except $t_i$. This facet $\theta$ is a part of a facet of
$K^\vee$. Indeed, otherwise, it would be in the interior of $K^\vee$ but
then points on the opposite side of $t_i$ can not be in any cone that contains
$\deg^\vee$. This contradicts centrality.

\medskip

The dual face of $\theta$ is generated by
some point $m\in K_{(1)}$. We must have $\la m,t_i \ra \neq 0$, since otherwise
$m$ is orthogonal to all of $\sigma$ and must be $0$. Therefore, we have
$\la m,t_i \ra=1$ and $\la m,n_j \ra=0$ for all $j$ by the definition of the
dual face. Thus this $m$ gives a monomial $F_m$ which is nonzero on $z$. This is a
contradiction.
\end{proof}

The upshot of the above discussion is the desired regularity property of $f=(f_1, \ldots, f_{k-r}) \in H^0(U_{\overline \Sigma}, \oplus_{i=1}^{k-r}\Oo_{\overline \Sigma})$ on $U_{\overline \Sigma}$ (see Section \ref{sec: construction} for notation and \eqref{eq: chi = lambda} for the fact that $f_i$ is $\hat G$-equivariant).

\begin{proposition}\label{prop: regularity}
For a general coefficient function $c$, $f$ is a regular section on $U_{\overline \Sigma}$ in the sense of Theorem \ref{thm: Hirano}, that is, $Y$ in \eqref{eq: Y} is a complete intersection of codimension $k-r$.
\end{proposition}
\begin{proof}
By Proposition \ref{prop: base point free}, the linear system $L_{i}$ is base point free. Hence for general coefficients, 
\[
Y= \bigcap_{1 \leq i \leq k-r} \{f_i = 0\} \subset U_{\overline \Sigma}
\]is a complete intersection by Bertini's theorem. 
\end{proof}

Recall that by definition \eqref{eq: bar N}, 
\[
\overline N  = N /\Zz s_1 + \cdots + \Zz s_{2r}+ \Zz \deg^\vee+ \Zz
t_1+\cdots+ \Zz t_{k-r}. 
\] The Cox open set $U_{\overline\Sigma}$ can be viewed as $\hat G = \overline G \times H$ invariant variety with $H$ acts trivially. We define two $\hat G$-equivariant locally free sheaves $\Ff_L$ and $\Ff_Q$ on $U_{\overline\Sigma}$.
 Let $\Ff_L  :=\oplus_{i=1}^{k-r} \Oo_{U_{\Sigma}}$ be the rank $k-r$ locally free sheaf associated to those $t_i$, or the linear part of the potential. Let $\Ff_Q:=\oplus_{i=1}^{2r} \Oo_{U_{\overline\Sigma}}(\chi_i)$ be the rank $2r$ locally free sheaf associated to $s_i$, or the quadric part of the potential, where $\chi_i$ is the character $\chi_i(\lambda) = \lambda(\zz(s_i))$.
 
 \medskip
 
 The vector bundle $V_{U_{\overline\Sigma}}(\Ff_Q)$ has a $\hat G = \overline G \times H$ action given by
 \begin{equation}\label{eq: G action}
(\bar g, t) \times (\bar x, \zz(s_i)) \mapsto (\bar g \cdot \bar x, t \cdot \zz(s_i)),
 \end{equation} where $\bar g \in  \overline G$ acts on $\bar x \in U_{\overline\Sigma}$ by the action of $\overline G$ on $\overline U_{\overline\Sigma}$, and $t \in H \cong \Cc^*$ by $t \cdot \zz(s_i) = t \zz(s_i), 1 \leq i \leq 2r$.  Then $C_2(\zz)$ is a section of $H^0(U_{\overline\Sigma}, {\rm Sym}^2\Ff_Q)$, and according to the discussion in Section \ref{sec: construction}, it is a $\hat G$-semiinvariant with character $\chi$. Similarly, the vector bundle $V_{U_{\overline\Sigma}}(\Ff_L(\chi) \oplus \Ff_Q)$ has $\hat G$-action given by 
 \[
 (\bar g, t) \times (\bar x, \zz(t_j), \zz(s_i)) \mapsto (\bar g \cdot \bar x, t\cdot \zz(t_j), t \cdot \zz(s_i)),
 \] where $\bar g \cdot \bar x, t \cdot \zz(s_i)$ are the same as \eqref{eq: G action}, and $t\cdot \zz(t_j) = t^2 \zz(t_j)$. Recall that by \eqref{eq: chi = lambda}, we have $\chi(\lambda) = \lambda(\zz(t_i))$ which amounts to multiple by $t^2$ under the identification $\hat G =\overline G \times H$.  Then $f=(f_i)_i$ is an $\hat G$-equivariant section of $\Ff_L$ according to the discussion in Section \ref{sec: construction}, and $C_1= \sum_{i=1}^{k-r} \zz(t_i) f_i$ in \eqref{eq: decomposition of C_1} is $\hat G$-semiinvariant with character $\chi$.
 
 \medskip
 
 By the previous discussion, we have 
 \begin{equation}\label{eq: identify stacks}
 [U_{\Sigma}/\hat G] \cong [V_{U_{\overline\Sigma}}(\Ff_L(\chi) \oplus \Ff_Q)/ \hat G].
 \end{equation} We also write $C_2$ for the pullback of $C_2$ to this vector bundle. The zero loci of $f$ is exactly $Y$ and $C_2(\zz)$ restricts to $V(\Ff_Q|_Y)$ is a quadric section, and hence associated to a sheaf of even Clifford algebra whose pullback to $\Ss = [Y/\overline G]$ is $\Bb_0$. 
 
\begin{remark}\label{rmk: notation}
We slightly abuse notation above: technically, the twist of $\chi$ should be on the locally free sheaf $\Ee$ which is the pullback of $\Ff_L$ over the morphism $V_{U_{\overline\Sigma}}(\Ff_Q) \to V_{U_{\overline\Sigma}}$. Thus, $V_{U_{\overline\Sigma}}(\Ff_L(\chi) \oplus \Ff_Q)$ should be written as $V_{V(\Ff_Q)}(\Ee(\chi))$ according to the notation in Theorem \ref{thm: Hirano}.
\end{remark}

We need \emph{flatness} assumption for the quadric fibration defined by $C_2(\zz)$:
\begin{equation}\label{eq: flatness}
\emph{The quadric fibration $\{C_2(\zz)= 0 \} \subset V_{U_{\overline\Sigma}}(\Ff_Q)$ to $U_{\overline\Sigma}$ is flat.}
\end{equation}

\begin{remark}\label{rm: flatness}
As pointed out in \cite{BL16} Remark 5.4. This flatness assumption is crucial for establishing the expected derived equivalence of pure Clifford double mirrors. There are examples that this no longer holds when the fibration is not flat. On the other hand, since the fibration is defined by a single equation, the geometric criterion for flatness is that all of the fibers are hypersurfaces.
\end{remark}

The following result is proved in \cite{BL16} Theorem 6.1 which relies on the results of \cite{Kuz08, BDFIK14}. Alternatively, one can use equivariant version of \cite{Kuz08} Theorem 4.2 and \cite{Orl99} Theorem 16 to obtain the same result.

\begin{proposition}\label{prop: Clifford to MF}
Under the flatness assumption \eqref{eq: flatness}, the derived category  $\D^b(\Ss, \Bb_0)$ is equivalent to the derived matrix factorization category $\Dcoh_{\hat G}(V(\Ff_Q|_Y), \chi, C_2(\zz))$.
\end{proposition}

\begin{proposition}\label{prop: Hir16}
Under the same assumption as before, there exists derived equivalence 
\[\D^b(\Ss, \Bb_0) \cong \Dcoh_{\hat G}(V_{U_{\overline\Sigma}}(\Ff_L(\chi) \oplus \Ff_Q), \chi, C(\zz)).\]
\end{proposition}
\begin{proof}
By Proposition \ref{prop: Clifford to MF}, we have 
\[\D^b(\Ss, \Bb_0) \cong \Dcoh_{\hat G}(V(\Ff_Q|_Y), \chi, C_2(\zz)).\]

The zero loci of $f=(f_i)_i$ on $V_{U_{\overline\Sigma}}(\Ff_Q)$ is $V(\Ff_Q|_Y)$, and the associated $\chi$-semiinvariant function $V_{U_{\overline\Sigma}}(\Ff_L(\chi) \oplus \Ff_Q) \to \mathbb{A}^1$ is exactly $C_1(\zz)$. By Proposition \ref{prop: regularity}, we know $C_1(\zz)$ is regular. Then by Theorem \ref{thm: Hirano} (see Remark \ref{rmk: notation}), there exists equivalence
\[\begin{split}
&\Dcoh_{\hat G}(V(\Ff_Q|_Y), \chi, C_2(\zz)) \\
\cong &\Dcoh_{\hat G}(V_{U_{\overline\Sigma}}(\Ff_L(\chi) \oplus\Ff_Q), \chi, C_1(\zz) + q^* C_2(\zz)),
\end{split}
\] where $q: V_{U_{\overline\Sigma}}(\Ff_L(\chi) \oplus\Ff_Q) \to V_{U_{\overline\Sigma}}(\Ff_Q)$ is the natural projection. However, $q^*C_2(\zz)$ is exactly the same as $C_2(\zz)$ restricted to $V_{U_{\overline\Sigma}}(\Ff_L(\chi) \oplus\Ff_Q)$ as there are no linear terms in $C_2(\zz)$. Finally, by the decomposition \eqref{eq: decomposition of C}, $C(\zz) = C_1(\zz) + C_2(\zz)$, we have the desired equivalence.
\end{proof}

By the same argument, we can establish the derived equivalence
\[
\D^b(\tilde \Ss, \tilde \Bb_0) \cong \Dcoh_{\hat G}(V_{U_{\overline\Sigma}}(\tilde \Ff_L(\chi) \oplus \tilde \Ff_Q), \chi, C(\zz)),
\] where $\tilde{(\  )}$ represent the analogous construction from the decomposition $\deg^\vee=\tilde t_1 + \ldots + \tilde r_{k-l}+\frac 1 2 (\tilde s_1 + \ldots + \tilde s_{2l})$, and simplicial fan $\tilde \Sigma$ satisfying centrality condition. We emphasize that $C(\zz)$ is the restriction of the same potential function \eqref{eq: potential}
\[
C({\bf z})=\sum_{m\in K_{(1)} }c(m)  \prod_{n\in
K^\vee_{(1)}} \zz(n)^{\la m,n\ra},
\] hence we use the same symbol. Moreover, by the definition of $\hat G$ (see \eqref{eq: G hat}) and $\chi$ , they do not depend on decompositions.

\medskip

Next, we relate the two matrix factorization categories by the result of \cite{BFK12}.

\begin{proposition}\label{prop: derived equivalent by VGIT}
There exists derived equivalence \[\Dcoh_{\hat G}(V_{U_{\overline\Sigma}}( \Ff_L(\chi) \oplus  \Ff_Q), \chi, C(\zz))\cong \Dcoh_{\hat G}(V_{U_{\overline{\tilde \Sigma}}}(\tilde \Ff_L(\chi) \oplus \tilde \Ff_Q), \chi, C(\zz)).\]
\end{proposition}

\begin{proof}
This result is proved in \cite{BL16} Theorem 3.2. By \eqref{eq: identify stacks}, the category $D_B(K,c;\Sigma)$ therein is exactly 
\[\Dcoh_{\hat G}(U_{\Sigma}, \chi, C(\zz)) \cong \Dcoh_{\hat G}(V_{U_{\overline\Sigma}}( \Ff_L(\chi) \oplus  \Ff_Q), \chi, C(\zz))\] under current notation.  Moreover, $\Dcoh_{\hat G}(V_{U_{\overline{\tilde \Sigma}}}(\tilde \Ff_L(\chi) \oplus \tilde \Ff_Q), \chi, C(\zz)) \cong D_B(K,c;\tilde \Sigma)$ for the same reason. Because $\Sigma, \tilde \Sigma$ are simplicial fans with support $K^\vee$, they corresponds to chambers of the secondary fan, and are connected by simple wall-crossing. By the definition of Gorenstein cone, $K_{(1)}$ lies on the hyperplane $\la -, \deg^\vee\ra=1$, which is exactly the condition needed for the derived equivalence. For details, one can consult the proof in \cite{BL16} Theorem 3.2, or \cite{FK16} Theorem 4.4.
\end{proof}

\begin{remark}
A far more general version of this result is proved by Ballard, Favero, Katzarkov \cite{BFK12} and Halpern-Leistner \cite{HL15} independently, which clarifies the earlier work of Herbst and Walcher \cite{HW12}.  The version stated above had appeared in \cite{FK16} (Theorem 4.4), where the derived matrix factorization category $D_B(K,c;\Sigma)$ is equivalent to $\hat G$-equivariant singular derived category $\{C=0\}$ therein due to the smoothness of $U_{\Sigma}$. Otherwise, this needs to be replaced by the $\hat G$-equivariant relative singular category (see \cite{EP15}).
\end{remark}

Put above results together, we have the desired equivalence:

\begin{theorem}\label{thm: main}
Under the flatness assumption \eqref{eq: flatness} on quadric fibrations, the general Clifford double mirrors $(\Ss, \Bb_0)$ and $(\tilde \Ss, \tilde \Bb_0)$ are derived equivalent.
\end{theorem}

\begin{proof}
By Proposition \ref{prop: Hir16}, we have 
\[\begin{split}
&\D^b(\Ss, \Bb_0) \cong \Dcoh_{\hat G}(V_{U_{\overline\Sigma}}( \Ff_L(\chi) \oplus  \Ff_Q), \chi, C(\zz))\\
&\D^b(\tilde \Ss, \tilde \Bb_0) \cong \Dcoh_{\hat G}(V_{U_{\overline{\tilde \Sigma}}}(\tilde \Ff_L(\chi) \oplus \tilde \Ff_Q), \chi, C(\zz)).
\end{split}
\] By Proposition \ref{prop: derived equivalent by VGIT}, we have $\D^b(\Ss, \Bb_0) \cong \D^b(\tilde \Ss, \tilde \Bb_0)$.
\end{proof}

\begin{remark}
We assume the existence of quadric fibration and its flatness in Theorem \ref{thm: main}. Hence, $r>0$ in the decomposition \eqref{eq: decomposition of degree}. When $r=0$, the associated noncommutative variety becomes the commutative Calabi-Yau variety (stack) $\Yy$ which is the Batyrev-Borisov complete intersection in the Fano toric variety. We still have derived equivalence between $\D^b(\Yy)$ and $\D^b(\Ss, \Bb_0)$ assuming the later having flat quadric fibration. This is because, $\D^b(\Yy)$ is also equivalent to $\Dcoh_{\hat G}(U_{\Sigma'}, \chi, C(\zz))$ by Shipman and Isik's result \cite{Shi12, Isi13} (notice that due to the flatness assumption in Hirano's Theorem, we cannot apply it here). Therefore, by Proposition \ref{prop: derived equivalent by VGIT}, we still have $\D^b(\Yy) \cong \D^b(\Ss, \Bb_0)$.
\end{remark}

\begin{remark}[Hochschild homology]
Kuznetsov had defined Hochschild homology for any semi-orthogonal component of the derived category of a variety \cite{Kuz09}. In particular, he showed that if $\Aa$ is a semi-orthogonal component of both $\D^b(X)$ and $\D^b(Y)$, then this definition does not depend on $X,Y$. Because we have $\D^b(\Ss, \Bb_0) = \D^b(\Yy)$, in particular, we have $HH_\bullet(\Ss, \Bb_0)$ defined  and it is the same as $HH_\bullet(\Yy)$ by \cite{Kuz09}. It is a general fact that  $HH_k(\Yy)$ is isomorphic to $\sum_{k=q-p}H^q(\Yy, \Omega_\Yy^p)$ under appropriate interpretations. Hence, the problem of verifying the equivalence of Hochschild homologies of Clifford double mirrors becomes to give a ``geometric '' definition of Hochschild homology of $(\Ss, \Bb_0)$, and then check it coincides with Kuznetsov's categorical definition. It seems that such definition had already appeared in noncommutative geometry (for example \cite{Giz05} \S 5).
\end{remark}

\begin{remark}[Hodge number]
Yongbin Ruan pointed out that it may be possible to study the relation between Hodge numbers of such noncommutative double mirrors by gauged linear sigma model where such relations had been well established.  
\end{remark}

\section{Examples}\label{sec: examples}

One type of examples of general Clifford double mirrors can be obtained as follows. Suppose there are two reflexive Gorenstein cones $K_1, K_2$ in the lattices $M_1, M_2$ respectively. Suppose their dual cones are $K_1^\vee, K_2^\vee$ in $M_1^\vee, M_2^\vee$ with degree elements $\deg^\vee_1, \deg_2^\vee$ respectively. Then there is rational polyhedral cone
\begin{equation}\label{eq: direct sum of cones K}
K:=\{(a; b) \in (M_1 \oplus M_2)_\Rr \mid a \in K_1, b \in K_2 \}.
\end{equation}It is a reflexive cone with degree element $\deg^\vee = (\deg^\vee_1; \deg^\vee_2)$. Its dual cone is exactly
\begin{equation}\label{eq: direct sum of cones K^vee}
K^\vee=\{(a^\vee; b^\vee) \in (M^\vee_1 \oplus M^\vee_2)_\Rr \mid a^\vee \in K^\vee_1, b^\vee \in K^\vee_2 \},
\end{equation} and hence $ K,  K^\vee$ is a pair of reflexive Gorenstein cones whose index is $\la \deg_1, \deg^\vee_1 \ra + \la \deg_2, \deg^\vee_2\ra = r_1 + r_2$.

\medskip

Now suppose the ``pure'' double mirror phenomenon happens for the reflexive Gorenstein cones $K_2, K_2^\vee$. To be precise, this means we have 
\begin{equation}\label{eq: decomposition for deg_1^vee}
\deg_2^\vee = \frac 1 2 (s_1 + \cdots + s_{2r_2}) = \frac 1 2 (\tilde s_1 + \cdots + \tilde s_{2r_2}),
\end{equation} that is we have \eqref{eq: decomposition of degree} without $t$-s. Moreover, we also assume the centrality conditions for triangulations of $K_2$. If there exists $\deg^\vee_1 = t_1 + \cdots + t_{r_1}$, then their combinations give a representations of $\deg^\vee$ as 
\begin{equation}\label{eq: representation of deg^vee}
\begin{split}
\deg^\vee &= (t_1;\0) + \cdots + (t_{r_1};\0) + \frac 1 2 \left((\0;s_1) + \cdots + (\0;s_{2r_2})\right)\\
&= (t_1;\0) + \cdots + (t_{r_1};\0) + \frac 1 2 \left((\0;\tilde s_1) + \cdots + (\0;\tilde s_{2r_2})\right).
\end{split}
\end{equation} Moreover, the regular triangulations $\Sigma_1$ and $\Sigma_2, \tilde \Sigma_2$ give triangulations 
$\Sigma$ and $\tilde \Sigma$, where $\Sigma = \{\sigma = (\sigma_i, \sigma'_j) \mid  \sigma_i \in \Sigma_1, \sigma'_j \in \Sigma_2 \}$ and $\tilde \Sigma =  \{\tilde\sigma = (\sigma_i, \tilde\sigma_j) \mid  \sigma_i \in \Sigma_1, \tilde \sigma_j \in \tilde \Sigma_2 \}$. The two triangulations $\Sigma$ and $\tilde \Sigma$  are regular and satisfy the centrality condition \eqref{eq: centrality}. As a result, we obtain a general Clifford double mirror from the pure ones.

\medskip

Because the $t$-part in two expressions of \eqref{eq: representation of deg^vee} are the same, we can view the Clifford double mirror obtained above as a family version of the pure Clifford double mirror parametrized by the complete intersection defined by $\deg_1^\vee = t_1 + \cdots +t_{r_1}$.

\medskip

The picture becomes more complicated when there are multiple ways to express $\deg^\vee_1$ which corresponding to double mirrors (no matter in the commutative sense or the Clifford sense). For example, if
\[
\deg_1^\vee = t_1 + \cdots +t_{r_1} = \tilde t_1 + \cdots +\tilde t_{r_1},
\] then we will obtain a commutative double mirror (see \cite{Li13} Theorem 3.4). Combining with \eqref{eq: decomposition for deg_1^vee}, we have general Clifford decompositions as \eqref{eq: representation of deg^vee}. However, the base of families of general Clifford double mirrors are different complete intersections (although they are birational in some situations, see \cite{Li13} Theorem 4.10). As one can generalize this easily, when $\deg^\vee_1$ also admits a pure Clifford decomposition
\[
\deg_1^\vee = \frac 1 2 (\xi_1 + \cdots + \xi_{2r_1}),
\] we can obtain decompositions
\begin{equation}\label{eq: representation of deg^vee}
\begin{split}
\deg^\vee &= (t_1;\0) + \cdots + (t_{r_1};\0) + \frac 1 2 \left((\0;s_1) + \cdots + (\0;s_{2r_2})\right)\\
&= \frac 1 2 \left((\xi_1;\0) + \cdots + (\xi_{2r_1};\0)  + (\0;\tilde s_1) + \cdots + (\0;\tilde s_{2r_2})\right).
\end{split}
\end{equation} In some sense, the general Clifford variety for the later decomposition is itself ``parametrized'' by a noncommutative variety.

\medskip

Perhaps, the most simple example of the above kind can be built upon the anticanonical hypersurface associated to the Fano toric variety defined by a reflexive polytope. For example (cf. \cite{BL16} Section 9.4), considering the $2$ dimensional reflexive polytope \[\Delta=\Conv\{(1,1),(1,-1),(-1,-1),(-1,1)\} \subset (M_1)_\Rr,\] whose dual polytope is
\[\Delta^\vee =
\Conv\{(1,0),(0,-1),(-1,0),(0,1)\} \subset (M_1^\vee)_\Rr.\] Then a pair of reflexive Gorenstein cones can be associated to these polytopes
\[
K_1 = \{(a;a\cdot\Delta) \mid a \geq 0\} \subset (M_1)_\Rr,\quad K_1^\vee =
\{(b;b\cdot\Delta^\vee) \mid b \geq 0\} \subset (M_1^\vee)_\Rr.
\] The degree element $\deg_1^\vee$ can be written in two different ways:
\[
\deg_1^\vee = (1;\mathbf{0}) = \frac{1}{2}(s_1 + s_2),
\] where $s_1 = (1,-1,0), s_2 = (1,1,0)$. Let us denote $(1;\mathbf{0})$ by $t$ for consistency.

\medskip

Let $E$ be an one dimensional variety defined by the expression $\deg_1^\vee = t$. It is an elliptic curve in $\mathbf{P}_{\Delta^\vee} = \Pp^1 \times \Pp^1$ given by equation
\begin{equation}\label{eq: elliptic curve}
\begin{split}
f = &a_{11}x^{-1}y + a_{12}y + a_{13}xy\\
+& a_{21}x^{-1} + a_{22} + a_{13}x\\
+& a_{31}x^{-1}y^{-1} + a_{32}y^{-1} + a_{33}xy^{-1},
\end{split}
\end{equation} where $a_{ij} \in \Cc$ are generically chosen  coefficients.

\medskip

There is a Clifford double mirror $(\Ss, \Bb_0)$ associated to $\deg_1^\vee = s_1 + s_2$. To be precise, $\overline{M_1^\vee} = \Zz^3 / \Zz^3 \cap(\Rr s_1 + \Rr s_2)  \cong
\Zz$, and $\Theta  = \overline{(K_1^{\vee})_{(1)}} = \Conv(-1,1)$, the
toric stack $\mathbf{P}_\Theta$ is actually the
smooth toric variety $\Pp^1$. By straightforward computations, one can find that among the lattice points of $(K_1)_{(1)}$, elements in $\{(1,-1,1),(1,-1,0),(1.-1,-1)\}$ pairing with $s_1$ equal to $2$; elements in $\{(1,1,1),(1,1,0),(1,1,-1)\}$ pairing with $s_2$ equal to $2$; and elements in $\{(1,0,1),(1,0,0),(1,0,-1)\}$ pairing with  both $s_1, s_2$ equal to $1$. If we use $z_1, z_2$ for the coordinate of the vector bundle, and $t$ for the coordinate of the base $S=\Pp^1$, then the quadric fibration can be written as
\begin{equation}
\begin{split}
C(\zz) =  &a_{11} z_1^2t  + a_{21}z_1^2 +
a_{31}z_1^2t^{-1}\\
+ &a_{12}z_1z_2t + a_{22}z_1z_2
+a_{32}z_1z_2t^{-1}\\
+ & a_{13}z_2^2t + a_{23}z_2^2 + a_{33}z_2^2t^{-1},
\end{split}
\end{equation} where $a_{ij}$ are the same coefficients as \eqref{eq: elliptic curve}. This is a family of quadrics parametrized by $t$, and because their coefficients are chosen generically, the corank of this quadric is less than $2$ for any $t$. Hence the derived category of $(\Ss, \Bb_0)$ is equivalent to a commutative variety $\tilde \Ss$, which is a ramified double cover of $\Ss$ (see \cite{Kuz08} Corollary 3.14). Moreover, the ramification loci is determined by where the quadric degenerates, that is,
\[(a_{11}t + a_{21} +
a_{31}z^{-1})(a_{13}t + a_{23} + a_{33}t^{-1}) - \frac{1}{4}(a_{12}t
+ a_{22} +a_{32}t^{-1})=0.\]

Hence, the ramification loci on $\Cc\Pp^{1}$ is exactly the same as those of $E$ over
$\Pp^1$. This shows that $\tilde \Ss, E$ are in fact isomorphic
elliptic curves. In particular, $\D^b(\Ss, \Bb_0) \cong \D^b(\tilde \Ss) \cong \D^b(E)$.

\medskip

Now, we can consider two copies of this example as discussed before. Let $M_1 = M_2 , M^\vee_1 = M^\vee_2$ and $K_1 = K_2$. Then $ K,  K^\vee$ constructed in \eqref{eq: direct sum of cones K}, \eqref{eq: direct sum of cones K^vee}  are reflexive Gorenstein cones in rank 6 lattices. There are three different ways to express $\deg^\vee$ which give double mirrors,
\[\begin{split}
\deg^\vee = & (t_1; \0) + (\0; t_1)\\
=& (t_1; \0) + \frac 1 2 \left((\0; s_1) + (\0; s_2)\right)\\
=& \frac 1 2 \left((s_1; \0)+ (s_2 ; \0) + (\0; s_1) + (\0; s_2)\right).
\end{split}
\] The first expression corresponds to Batyrev-Borisov complete intersection, which is a product of two elliptic curves $E \times E$ according to the previous discussion; the third expression corresponds to the ``pure'' Clifford double mirror  $(\Ss_K, (\Bb_K)_0)$ explored in \cite{BL16} (or viewed as a special case of general Clifford double mirror). Their derived equivalence is a consequence of the Theorem 6.3 of \cite{BL16}. The second expression exhibits the general Clifford double mirror $(\Ss_{\rm gen}, (\Bb_{\rm gen})_0)$ considered in this paper, where the hypersurface $\Ss_{\rm gen} = E$ defined by $(t_1; \0)$ ``parametrizes'' pure Clifford double mirror $(\Ss_t, (\Bb_t)_0) = (\Ss, \Bb_0)$. We had shown that $E$ and $(\Ss, \Bb_0)$ have equivalent derived categories. Therefore, it is reasonable to have the derived equivalence 
\[
\D^b(E \times E) \cong \D^b(\Ss_{\rm gen}, (\Bb_{\rm gen})_0) \cong \D^b(\Ss_K, (\Bb_K)_0)
\] which is the consequence of Theorem \ref{thm: main}. 

\begin{remark}
In general, when the reflexive Gorenstein cone does not have above ``direct sum'' property,  the corresponding noncommutative varieties are predictably more complicated. It is interesting to classify the combinatoric data for general Clifford mirrors in low dimensions just as in the Batyrev-Borisov case. Moreover, it is desirable to use the above method to construct examples as \cite{Cal00b}, where an elliptic threefold without a section is derived equivalent to the twisted derived category of its (small resolution of) relative Jacobian. The twist there should be replaced by the even Clifford sheaf $\Bb_0$.
\end{remark}

\bibliographystyle{alpha}
\bibliography{bibfile}

\newcommand{\etalchar}[1]{$^{#1}$}
\begin{thebibliography}{CDH{\etalchar{+}}10}

\bibitem[Add09]{Add09}
Nicolas Addington.
\newblock The derived category of the intersection of four quadrics.
\newblock {\em arXiv:0904.1764}, 2009.

\bibitem[BB96]{BB96}
Victor~V Batyrev and Lev~A Borisov.
\newblock Mirror duality and string-theoretic {H}odge numbers.
\newblock {\em Inventiones mathematicae}, 126(1):183--203, 1996.

\bibitem[BB97]{BB94}
Victor~V. Batyrev and Lev~A. Borisov.
\newblock Dual cones and mirror symmetry for generalized {C}alabi-{Y}au
  manifolds.
\newblock In {\em Mirror symmetry, {II}}, volume~1 of {\em AMS/IP Stud. Adv.
  Math.}, pages 71--86. Amer. Math. Soc., Providence, RI, 1997.

\bibitem[BCS05]{BCS05}
Lev~A. Borisov, Linda Chen, and Gregory~G. Smith.
\newblock The orbifold {C}how ring of toric {D}eligne-{M}umford stacks.
\newblock {\em J. Amer. Math. Soc.}, 18(1):193--215 (electronic), 2005.

\bibitem[BDF{\etalchar{+}}14]{BDFIK14}
Matthew Ballard, Dragos Deliu, David Favero, M~Umut Isik, and Ludmil Katzarkov.
\newblock On the derived categories of degree $d$ hypersurface fibrations.
\newblock {\em arXiv:1409.5568}, 2014.

\bibitem[BFK12]{BFK12}
Matthew Ballard, David Favero, and Ludmil Katzarkov.
\newblock Variation of geometric invariant theory quotients and derived
  categories.
\newblock {\em arXiv:1203.6643}, 2012.

\bibitem[BL16]{BL16}
Lev Borisov and Zhan Li.
\newblock On {C}lifford double mirrors of toric complete intersections.
\newblock {\em arXiv:1601.00809}, 2016.

\bibitem[BN08]{BN08}
Victor Batyrev and Benjamin Nill.
\newblock Combinatorial aspects of mirror symmetry.
\newblock In {\em Integer points in polyhedra---geometry, number theory,
  representation theory, algebra, optimization, statistics}, volume 452 of {\em
  Contemp. Math.}, pages 35--66. Amer. Math. Soc., Providence, RI, 2008.

\bibitem[Bor13]{Bor13}
Lev~A. Borisov.
\newblock Berglund-{H}\"ubsch mirror symmetry via vertex algebras.
\newblock {\em Comm. Math. Phys.}, 320(1):73--99, 2013.

\bibitem[C{\u{a}}l02]{Cal00b}
Andrei C{\u{a}}ld{\u{a}}raru.
\newblock Derived categories of twisted sheaves on elliptic threefolds.
\newblock {\em J. Reine Angew. Math.}, 544:161--179, 2002.

\bibitem[CDH{\etalchar{+}}10]{CDHPS10}
Andrei C{\u{a}}ld{\u{a}}raru, Jacques Distler, Simeon Hellerman, Tony Pantev,
  and Eric Sharpe.
\newblock Non-birational twisted derived equivalences in abelian {GLSM}s.
\newblock {\em Comm. Math. Phys.}, 294(3):605--645, 2010.

\bibitem[CK99]{CK99}
David~A. Cox and Sheldon Katz.
\newblock {\em Mirror symmetry and algebraic geometry}, volume~68 of {\em
  Mathematical Surveys and Monographs}.
\newblock American Mathematical Society, Providence, RI, 1999.

\bibitem[CKP15]{CKP15}
Tom Coates, Alexander Kasprzyk, and Thomas Prince.
\newblock {F}our-dimensional {F}ano toric complete intersections.
\newblock {\em Proceedings. Mathematical, physical, and engineering
  sciences/the Royal Society}, 471(2175), 2015.

\bibitem[Cla14]{Cla13}
Patrick Clarke.
\newblock A proof of the birationality of certain {BHK}-mirrors.
\newblock {\em Complex Manifolds}, 1:45--51, 2014.

\bibitem[CLS11]{CLS11}
David~A. Cox, John~B. Little, and Henry~K. Schenck.
\newblock {\em Toric varieties}, volume 124.
\newblock American Mathematical Soc., 2011.

\bibitem[CR11]{CR11}
Alessandro Chiodo and Yongbin Ruan.
\newblock L{G}/{CY} correspondence: the state space isomorphism.
\newblock {\em Adv. Math.}, 227(6):2157--2188, 2011.

\bibitem[CT15]{CT14}
John~R. Calabrese and Richard~P. Thomas.
\newblock Derived equivalent {C}alabi--{Y}au threefolds from cubic fourfolds.
\newblock {\em Math. Ann.}, pages 1--18, 2015.

\bibitem[EP15]{EP15}
Alexander~I. Efimov and Leonid Positselski.
\newblock Coherent analogues of matrix factorizations and relative singularity
  categories.
\newblock {\em Algebra Number Theory}, 9(5):1159--1292, 2015.

\bibitem[FK16a]{FK16b}
David Favero and Tyler~L Kelly.
\newblock Derived categories of bhk mirrors.
\newblock {\em arXiv preprint arXiv:1602.05876}, 2016.

\bibitem[FK16b]{FK16}
David Favero and Tyler~L Kelly.
\newblock Proof of a conjecture of {B}atyrev and {N}ill.
\newblock {\em to appear in Amer. J. Math}, 2016.

\bibitem[Gin05]{Giz05}
Victor Ginzburg.
\newblock Lectures on noncommutative geometry.
\newblock {\em arXiv preprint math/0506603}, 2005.

\bibitem[HD15]{HD15}
Andrew Harder and Charles~F Doran.
\newblock Toric degenerations and the {L}aurent polynomials related to
  {G}ivental's {L}andau-{G}inzburg models.
\newblock {\em arXiv:1502.02079}, 2015.

\bibitem[Hir16]{Hir16}
Yuki Hirano.
\newblock Derived kn{\"o}rrer periodicity and {O}rlov's theorem for gauged
  {L}andau-{G}inzburg models.
\newblock {\em arXiv:1602.04769}, 2016.

\bibitem[HL15]{HL15}
Daniel Halpern-Leistner.
\newblock The derived category of a {GIT} quotient.
\newblock {\em J. Amer. Math. Soc.}, 28(3):871--912, 2015.

\bibitem[HW12]{HW12}
Manfred Herbst and Johannes Walcher.
\newblock On the unipotence of autoequivalences of toric complete intersection
  {C}alabi-{Y}au categories.
\newblock {\em Math. Ann.}, 353(3):783--802, 2012.

\bibitem[Isi13]{Isi13}
Mehmet~Umut Isik.
\newblock Equivalence of the derived category of a variety with a singularity
  category.
\newblock {\em Int. Math. Res. Not. IMRN}, (12):2787--2808, 2013.

\bibitem[Kel13]{Kel13}
Tyler~L. Kelly.
\newblock Berglund-{H}\"ubsch-{K}rawitz mirrors via {S}hioda maps.
\newblock {\em Adv. Theor. Math. Phys.}, 17(6):1425--1449, 2013.

\bibitem[Kon95]{Kon94}
Maxim Kontsevich.
\newblock Homological algebra of mirror symmetry.
\newblock In {\em Proceedings of the {I}nternational {C}ongress of
  {M}athematicians, {V}ol.\ 1, 2 ({Z}\"urich, 1994)}, pages 120--139.
  Birkh\"auser, Basel, 1995.

\bibitem[Kuz08]{Kuz08}
Alexander Kuznetsov.
\newblock Derived categories of quadric fibrations and intersections of
  quadrics.
\newblock {\em Adv. Math.}, 218(5):1340--1369, 2008.

\bibitem[Kuz09]{Kuz09}
Alexander Kuznetsov.
\newblock Hochschild homology and semiorthogonal decompositions.
\newblock {\em arXiv:0904.4330}, 2009.

\bibitem[Li13]{Li13}
Zhan Li.
\newblock On the birationality of complete intersections associated to
  nef-partitions.
\newblock {\em arXiv:1310.2310}, 2013.

\bibitem[Muk88]{Muk88}
Shigeru Mukai.
\newblock Moduli of vector bundles on {K}3 surfaces, and symplectic manifolds.
\newblock {\em Sugaku Expositions}, 1(2):138--174, 1988.

\bibitem[Orl09]{Orl99}
Dmitri Orlov.
\newblock Derived categories of coherent sheaves and triangulated categories of
  singularities.
\newblock In {\em Algebra, arithmetic, and geometry: in honor of {Y}u. {I}.
  {M}anin. {V}ol. {II}}, volume 270 of {\em Progr. Math.}, pages 503--531.
  Birkh\"auser Boston, Inc., Boston, MA, 2009.

\bibitem[Pos11]{Pos11}
Leonid Positselski.
\newblock Two kinds of derived categories, {K}oszul duality, and
  comodule-contramodule correspondence.
\newblock {\em Mem. Amer. Math. Soc.}, 212(996):vi+133, 2011.

\bibitem[Pri]{Pri}
Thomas Prince.
\newblock Ph.{D} {T}hesis.
\newblock {\em Imperial College London, In preparation.}

\bibitem[Shi12]{Shi12}
Ian Shipman.
\newblock A geometric approach to {O}rlov's theorem.
\newblock {\em Compos. Math.}, 148(5):1365--1389, 2012.

\bibitem[Sho14]{Sho12}
Mark Shoemaker.
\newblock Birationality of {B}erglund-{H}\"ubsch-{K}rawitz mirrors.
\newblock {\em Comm. Math. Phys.}, 331(2):417--429, 2014.

\end{thebibliography}

\end{document}